\RequirePackage{fix-cm}
\documentclass[pdflatex,sn-mathphys]{sn-jnl}

\usepackage{anyfontsize}
\usepackage{bm}
\usepackage{tikz-cd}
\usepackage{enumerate}
\usepackage{mathabx}

\usepackage{amsmath}
\usepackage{latexsym}
\usepackage{amsthm}
\usepackage{amssymb}
\usepackage{draftcopy}
\usepackage{comment}
\usepackage{mathabx}
\usepackage{matlab-prettifier}
\usepackage[T1]{fontenc}
\usepackage{algorithm}
\usepackage{algorithmicx}
\usepackage{algpseudocode,refcount}
\usepackage{mathtools}
\usepackage{centernot}
\usepackage{stmaryrd}
\SetSymbolFont{stmry}{bold}{U}{stmry}{m}{n}
\usepackage{comment}
\usepackage{caption} 
\usepackage{graphicx}%
\usepackage{multirow}%
\usepackage{amsthm}%
\usepackage{mathrsfs}%
\usepackage[title]{appendix}%
\usepackage{xcolor}%
\usepackage{textcomp}%
\usepackage{manyfoot}%
\usepackage{booktabs}%
\usepackage{algorithm}%
\usepackage{algorithmicx}%
\usepackage{algpseudocode}%
\usepackage{listings}%

\newcommand{\Z}{\mathbb{Z}}

\newcommand{\C}{\mathbb{C}}

\newcommand{\aut}{\text{Aut}}

\newcommand{\paf}{\text{PAF}}

\newcommand{\zl}{\Z_\ell}
\newcommand{\zlx}{\zl^\times}
\newcommand{\sdp}{\zl\rtimes_{\theta}\zlx}
\newcommand{\sdpq}{K\rtimes_{\theta}Q}

\newcommand{\bu}{\mathbf{u}}

\newcommand{\bv}{\mathbf{v}}

\newcommand{\bx}{\mathbf{x}}

\newcommand{\by}{\mathbf{y}}

\newcommand{\bone}{\mathbf{1}}
\newcommand{\bzero}{\mathbf{0}}

\newcommand{\cl}{\C^{\zl}}

\newtheorem{theorem}{Proposition}[section]

\newtheorem{definition}{Definition}[section]

\newtheorem{corollary}[theorem]{Corollary}
\newtheorem{lemma}[theorem]{Lemma}

\numberwithin{equation}{section}

\begin{document}

\title[Article Title]{Determining the group that sends each Legendre pair to an equivalent Legendre pair}

\author*[1]{\fnm{Dursun} \sur{Bulutoglu}}\email{dursun.bulutoglu@gmail.com}
\author[2]{\fnm{Daniel} \sur{Baczkowski}}\email{baczkowski@findlay.edu}
\author[1]{\fnm{Joshua} \sur{Yauney}}\email{joycreatedme@gmail.com}
\affil[1]{\orgdiv{Mathematics \& Statistics}, \orgname{Air Force Institute of Technology}, \orgaddress{\street{2950 Hobson Way}, \city{WPAFB}, \postcode{45433}, \state{OH}, \country{USA}}}
\affil[1]{\orgdiv{Mathematics \& Statistics}, \orgname{University of Findlay}, \orgaddress{\street{1000 N. Main St.}, \city{Findlay}, \postcode{45840}, \state{OH}, \country{USA}}}

\abstract{In this paper we determine the structure of the group of all operations that send each Legendre pair to an equivalent Legendre pair.}

\keywords {Cyclic shift; Decimation; Semidirect product; }

\pacs[MSC Classification]{05E18 20D40 20E22}

\maketitle

\section{Introduction}\label{sec1}

Let \(\bv\) be a vector indexed by \(\zl\). Then \(\paf(\bv,j)=\sum_{i\in \zl}v_iv_{i-j}\) for all \(j\in\zl\) is called the \emph{periodic autocorrelation function} (PAF) of $\bv$.
The concept of a Legendre pair  of length $\ell$ was first introduced in~\cite{fgs01} as a combinatorial object to construct
 Hadamard matrices of order $2\ell+2$ for each odd $\ell$.
\begin{definition}\label{defn:orig}
Two vectors, \(\bu\)~and~\(\bv\), in \(\{-1,1\}^{\zl}\) constitute a \emph{Legendre pair} (LP) of length $\ell$ if
$\bone^\top\bu=\bone^\top\bv$ and
\begin{equation}
\paf(\bu,j)+\paf(\bv,j)=-2\quad\forall\,j\in\zl\setminus\{0\}.\label{eqn:lppsdd}
\end{equation}
\end{definition}

First, we  fix some notation and basic facts that are adopted throughout the paper. Let $\bone$ be an all $1$s column vector of length $\ell$, $\bzero$ be an all $0$s column vector of length $\ell$, \(\zlx\) be the multiplicative group of \(\zl\).
 For any set $S \subseteq \mathbb{C}$, $S^{\zl}$ is the set of all $S$-valued vectors indexed by $\zl$. 
 We fix  \(\bu,\bv \in \mathbb{C}^{\zl}\), where \(\bu = [u_0,u_1,\ldots,u_{\ell-1}]^{\top}, \bv = [v_0,v_1,\ldots,v_{\ell-1}]^{\top}\). We use non-boldface letters for entries of vectors and matrices, lowercase boldface
 letters for vectors, and uppercase  boldface letters for matrices. All vectors are column vectors.
 For a set $A$ with
elements from a
group $G$, 
 $\langle a \mid a \in A \rangle$ is the group generated by elements of $A$. 
 We say that a set $S_1$ is smaller than $S_2$ if $S_1 \subset S_2$. 
 For a group $G$ acting on a set $S$ and $x \in S$, $G\,x$ is the orbit of $x$ under the action of $G$. 
 For two groups $G_1\leq G_2$ means $G_1$ is a subgroup of $G_2$.
 For  subgroups $K_1,K_2,\ldots,K_r$ of a group $G$, $\langle K_1, K_2,\ldots,K_r\rangle $ is the group generated by the elements of $K_1,K_2,\ldots,K_r$, i.e., the smallest subgroup of $G$ containing each  $K_i$ for $i\in \{1,\ldots,r\}$. 
 For subgroups $G_1,G_2,G_3$ of a group $G$, we have  $\langle G_1,  G_2,G_3 \rangle =\langle G_1, \langle G_2,G_3\rangle \rangle$. To see this, observe $\langle G_1,  G_2,G_3 \rangle \leq \langle G_1, \langle G_2,G_3\rangle \rangle$
 because  $\langle G_1, \langle G_2,G_3\rangle \rangle$ is a group containing $G_1,G_2$, and $G_3$. On the other hand, 
 $G_1 \leq \langle G_1,  G_2,G_3 \rangle $ and $\langle G_2,G_3\rangle  \leq \langle G_1,  G_2,G_3 \rangle $ as  $\langle G_1,  G_2,G_3 \rangle$ is a group containing $G_1$, $G_2$, and $G_3$ so it must contain the smallest group containing $G_2$ and $G_3$. Hence,  $\langle G_1, \langle G_2,G_3\rangle \rangle\leq \langle G_1,  G_2,G_3 \rangle$.

For $j \in \zl$, let $c_j:\mathbb{C}^{\zl}\rightarrow \mathbb{C}^{\zl}$ be the linear map such that $$\big(c_j(\bv)\big)_i=v_{i-j}$$ for each \(i\in\zl\). Then $c_j$ is called the  
\emph{cyclic shift by $j$} map or simply a \emph{cyclic shift}, and $c_j(\bv)$ is the cyclic shift of $\bv$ by $j \in \zl$.

 For a vector  \(\bv\) indexed by \(\zl\),
 a \emph{decimation}  of \(\bv\) by \(k\in\zlx\)  is the transformation  \(d_k:\mathbb{C}^{\zl} \rightarrow \mathbb{C}^{\zl}\) such that 
 \begin{equation}\label{eqn:decimate}
 \big(d_k(\bv)\big)_i=v_{ik^{-1}}
 \end{equation}
 for each \(i\in\zl\). 

We can compose cyclic shifts with decimations and get
\begin{equation}\label{eqn:acil}
   ((c_a\circ d_b) (\bv))_i=c_a(d_b(\bv))_i=v_{(i-a)b^{-1}} 
\end{equation}
for each $a\in \zl$, $b \in \zlx$, and $\bv \in \mathbb{C}^{\zl}$.
Let $D$ be the group generated by cyclic shifts $c_i$ for $i \in \zl$
and decimations $d_k$ for $k \in \zlx $, where the group operation is function composition. In Section~\ref{sec:equivalence}, we will show that $D$ acts on $\cl$.
The orbit of $\bu$ in $\cl$ under the action of 
$D$  is called the {\em decimation class} of $\bu$. Two vectors in the same decimation class
are called {\em equivalent}. 
(Being in the same orbit under the action of a group on a set is an equivalence relation on that set.)

It is clear from Definition~\ref{defn:orig} that if $(\bu,\bv)$ is an LP, then $(\bv,\bu)$ is also an LP. Then, 
for a pair of ordered vectors \((\bu,\bv)\), let $s$ be the {\em switch} operator defined by
$$s(\bu,\bv)=(\bv,\bu).$$

By using cyclic shifts, decimations, and the switch operator, we now define the equivalence concept for pairs of vectors in $\mathbb{C}^{\ell}$ so that the validity of equation~(\ref{eqn:lppsdd}) is preserved among equivalent pairs. 
Two pairs of vectors \((\bu,\bv)\) and \((\bu',\bv')\) in $\mathbb{C}^{\zl}\times \mathbb{C}^{\zl}$
 are called \emph{equivalent} if
\begin{equation}\label{eqn:equiv1}
(\bu',\bv')=\bigg(s^f\Big(d_k\big(c_i(\bu)\big),d_{k}\big(c_j(\bv)\big)\Big)\bigg)
\end{equation}
or 
\begin{equation}\label{eqn:equiv2}
(\bu',\bv')=\bigg(s^f\Big(d_k\big(c_i(\bu)\big),d_{-k}\big(c_j(\bv)\big)\Big)\bigg)
\end{equation}
for some \(k\in\zlx\), \(i,j\in\zl\), and \(f\in\{0,1\}\). 
In Section~\ref{sec:equivalence}, we will show that the equivalence defined on $\cl\times \cl$ is indeed an equivalence relation on $\cl\times \cl$.
It is well known that if
$(\bu,\bv)$ is an LP and $(\bu',\bv')$ is equivalent to $(\bu,\bv)$, then 
$(\bu',\bv')$ is also an LP~\cite{abh20}; for a proof, see~\cite{yauney24}. Because of this fact, the set of all LPs of length $\ell$ can be generated from a set of non-equivalent LPs by using equations~(\ref{eqn:equiv1}) and~(\ref{eqn:equiv2}). Hence, it suffices to find a set of all non-equivalent LPs of length $\ell$ to find all LPs of length $\ell$. Also, since the number of all non-equivalent LPs of length $\ell$ tends to be much smaller than the number of all LPs of length $\ell$~\cite{fgs01}, keeping only a set of all  non-equivalent LPs drastically decreases the disk space requirements.

One way an exhaustive search algorithm for LPs can be expedited is by narrowing its search space.
However, the main risk of narrowing the search space is that no LP of length $\ell$ exists in the narrowed space while an LP of length $\ell$ exists. However, if this narrowing is designed based on the equivalence of LPs in such a way that the narrowed search space contains at least
a set of all non-equivalent LPs, then the exhaustive search algorithm will find an LP of length $\ell$ if it exists. This is true because equation~(\ref{eqn:lppsdd})  is preserved among equivalent pairs of vectors in $\mathbb{C}^{\zl}\times \mathbb{C}^{\zl}$. Then, restricting the search to finding a set of LPs that includes all non-equivalent LPs  is a viable way to accomplish computational savings. 
Hence, 
the concept of equivalence of LPs plays a key role in designing faster LP finding methods by narrowing the search space.

The concept of equivalence is based on  decimations and cyclic shifts. So, we first study the group $D$ generated by all cyclic shifts and decimations.
Each element of the group $D$ generated by all cyclic shifts and decimations acts on each vector $\bv$ in $\mathbb{C}^{\zl}$ by permuting its entries. The orbits of this action constitute the decimation classes.
In Section~\ref{sec:equivalence}, we will first show that $D$ indeed acts on $\mathbb{C}^{\zl}$, 
and determine the structure of $D$. Then, we are going to construct the group $GG$ that sends each $(\bu,\bv)$ 
to another pair $(\bu',\bv')$ that is equivalent to it,  determine the structure of $GG$, and show that $GG$ indeed acts on vectors in $\mathbb{C}^{\zl}\times \mathbb{C}^{\zl}$.
The structure of $GG$ has never been studied before.
However, 
the fact that that $GG$ acts on LPs of length $\ell$ has been routinely stated or used by researchers in the field~\cite{fgs01,kkbatr23,arq,kk21}. 

\section{Determining the structure of $GG$}\label{sec:equivalence}
We first define the semidirect product of two groups that plays an essential
role in the study of equivalence of LPs.
For a given group $H$, let $\aut(H)$ be its group of all automorphisms.
Given two groups $H$, $K$ and a homomorphism $\theta: K \rightarrow  \aut(H)$, we define their semidirect product group $H\rtimes_\theta K$ to be the set of all ordered pairs $(h,k) \in H \times K$
with the group operation $(h_1,k_1)(h_2,k_2)=(h_1\theta_{k_1}(h_2),k_1k_2)$, i.e., $\theta$
takes $k \in K$
to the $H$-automorphism 
$\theta_{k}\in \aut(H)$~\cite{rotmanintro}. 
It is shown in~\cite{rotmanintro} that $H\rtimes_\theta K$ indeed satisfies the group axioms.
The following lemma follows from the 
definition of semidirect products and will be needed in the study of the equivalence of LPs.
\begin{lemma}\label{lem:generate}
Let $H$ be a group with the identity 
element $e_1$, and $K$ be a group with the
identity element $e_2$. 
 Let $S_H=\{h_1,h_2,\ldots, h_{\gamma_1}\}$  and  $S_K=\{k_1,k_2,\ldots, k_{\gamma_2}\}$ 
 be a set of generators for the groups $H$ and $K$, respectively. Let $\theta$ be a homomorphism $\theta: K \rightarrow  \aut(H)$.
 Then the elements $(h_1,e_2),(h_2,e_2),\ldots,  (h_{\gamma_1},e_2),
 (e_1,k_1),(e_1,$ $k_2),\ldots,(e_1,k_{\gamma_2})$ generate $H\rtimes_\theta K$.
\end{lemma}
\begin{proof}
Since $\theta$ is a homomorphism, $\theta_{e_2}$ is the identity automorphism in $\aut(H)$, and 
  $\theta_x \in \aut(H)$ implies $\theta_x(e_1)=e_1$. 
    Observe that for $a,b \in H$ and $x,y \in K$, 
  \begin{align}\label{eqn:gensdp}
  \begin{split}
      (a,e_2)(b,e_2)&=(a\theta_{e_2}(b),e_2e_2)=(ab,e_2),\\
  (e_1,x)(e_1,y)&=(e_1\theta_x(e_1),xy)=(e_1,xy),\\
   (h,e_2)(e_1,k)&=(h\theta_{e_2}(e_1),ke_2)=(he_1,ke_2)=(h,k).
   \end{split}
  \end{align}  
  For arbitrary $h \in H$ and $k \in K$, there exist 
  $h_{i_1},\ldots, h_{i_{r_1}} \in S_H$ and
  $k_{j_1},\ldots, k_{j_{r_2}} \in S_K$ such that $h= h_{i_1}\cdots h_{i_{r_1}}$ and $k=k_{j_1}\cdots k_{j_{r_2}}$. Then, by equations~(\ref{eqn:gensdp}),
 \begin{align*}
 (h_{i_1},e_2)(h_{i_2},e_2)\cdots  (h_{i_{r_1}},e_2)(e_1,k_{j_1})(e_1,k_{j_2})\cdots(e_1,k_{j_{r_2}})
 &=(h_{i_1}\cdots h_{i_{r_1}}, k_{j_1}\cdots k_{j_{r_2}})\\
 &=(h,k).
 \end{align*} 
\end{proof}

To prove that decimation classes are indeed equivalence classes, our next goal is to show that 
  the group $D$ generated by $d_k$'s for $k\in\zlx$ and $c_j$'s for \(j\in\zl\) with group operation function composition acts on $\mathbb{C}^{\zl}$. To accomplish this goal, we will first show that
  \(\sdp\) acts 
 on $\mathbb{C}^{\zl}$, where $\theta_k(i)=ki$
  for $k \in \zlx$ and $i \in \zl$,
 i.e., $\theta$ is the homomorphism that takes $k \in \zlx$ to the $\zl$-automorphism $i\xmapsto{\sigma_k} ki$ for each $i \in \zl$.
 Now, $\theta:\zlx\rightarrow \aut(\zl)$ is indeed a homomorphism because for each $k_1,k_2 \in \zlx$, $$\theta(k_1)\circ\theta(k_2)(i)=\sigma_{k_1}\circ\sigma_{k_2}(i)=k_1k_2i=\sigma_{k_1k_2}(i)=\theta(k_1k_2)(i)$$ for all $i \in \zl$. Hence, 

$$\theta(k_1k_2)=\sigma_{k_1k_2}=\theta(k_1)\circ\theta(k_2)=\sigma_{k_1}\circ\sigma_{k_2}.$$
Let $id$ be the identity automorphism of $\zl$. Then, $\sigma_k$ for each $k \in \zlx$ is an invertible map from $\zl$ to $\zl$ because $\sigma_k(i)\in \zl$, and 
$$\sigma_k\circ\sigma_{k^{-1}}=\sigma_{k^{-1}}\circ\sigma_k=id,$$  implying $\sigma_k^{-1}=\sigma_{k^{-1}}$.
Moreover, the homomorphism property of $\sigma_k$ follows because 
$\sigma_k(i_1+i_2)=ki_1+ki_2=\sigma_k(i_1)+\sigma_k(i_2)$ for each $i_1,i_2 \in \zl$. Hence, $\sigma_k$ is a  $\zl$-automorphism. 
For the remainder of the paper, unless otherwise stated, let  $\theta:\zlx\rightarrow \aut(\zl)$ be defined as above.
 
 Proving the following lemma is the first step in showing that  \(\sdp\) acts 
 on $\mathbb{C}^{\zl}$.
 \begin{lemma}\label{lem:acts}
 The group \(\sdp\) acts on  \(\zl\) by
\[(a,b)g=a+bg\] 
for \((a,b)\in\sdp\) and \(g\in\zl\). 
\end{lemma}
\begin{proof}
 The identity of     \(\sdp\) is $(0,1)$, and 
$(0,1)g=0+g=g$ for each $g \in \zl$.
Moreover for $(a_1,b_1), (a_2,b_2) \in  \sdp$, 
$$(a_1,b_1)((a_2,b_2)g)=(a_1,b_1)(a_2+b_2g)=
a_1+b_1a_2+b_1b_2g,$$
and 
\begin{equation*}\label{eqn:multiplicationgen} 
((a_1,b_1)(a_2,b_2))g=(a_1+\theta_{b_1}(a_2), b_1b_2)g=
a_1+b_1a_2+b_1b_2g.
\end{equation*}
Hence, $(a_1,b_1)((a_2,b_2)g)=  ((a_1,b_1)(a_2,b_2))g$, and  \(\sdp\) acts on  \(\zl\).
\end{proof}
The following lemma shows that the action of  \(\sdp\) on \(\zl\) induces an action
 of   \(\sdp\) on vectors in $\mathbb{C}^{\zl}$ by taking $\psi$ to be the identity isomorphism.
 \begin{lemma}\label{lem:Danc}
 Let $\psi$ be an isomorphism $\psi:\sdp\rightarrow T$ from $\sdp$ to a group $T$.
Then the action of  \(\sdp\) on \(\zl\) induces an action
 of   $T$ on vectors in $\mathbb{C}^{\zl}$, 
 where for each $t\in T$ such that $(a,b)=\psi^{-1}(t) \in \sdp$, $\bv \in \cl$, $g \in \zl$ we have
 \begin{equation*}\label{eqn:genact}
 t(\bv)_g=\psi^{-1}(t)(\bv)_g =(a,b)(\bv)_g =v_{(a,b)^{-1}g}= v_{{{(g-a)b^{-1}}}}=((c_a \circ d_b)(\bv))_g
 \end{equation*}
  and 
 $$ 
t\bv =\psi^{-1}(t)(\bv)= (a,b)\bv =(c_a \circ d_b)(\bv).
 $$
 \end{lemma}
 \begin{proof}
To prove that there is an induced action of $T$  on $\cl$, let $t_1,t_2 \in T$ be such that
 $(a_1,b_1)=\psi^{-1}(t_1)$ and $ (a_2,b_2)=\psi^{-1}(t_2) \in  \sdp$. Then 
 \begin{align*}\label{eqn:actionzl}
 \begin{split}
(t_1(t_2\bv))_g=((a_1,b_1)((a_2,b_2)\bv))_g =&\,\,   (a_2,b_2)(\bv_{(a_1,b_1)^{-1}g})\\=&\,\,v_{(a_1,b_1)^{-1}((a_2,b_2)^{-1}g)}=
 v_{((a_1,b_1)^{-1}(a_2,b_2)^{-1})g}\\=&\,\,((a_1,b_1)(a_2,b_2))\bv_g\\=&\,\,
 (\psi^{-1}(t_1)\psi^{-1}(t_2))\bv_g\\=&\,\,
  (\psi^{-1}(t_1t_2))\bv_g\\=&\,\,
  (t_1t_2)\bv_g
 \end{split}
 \end{align*}
 implying 
 \begin{equation*}\label{eqn:actsonv}
(t_1(t_2\bv))= (a_1,b_1)((a_2,b_2)\bv)=((a_1,b_1)(a_2,b_2))\bv=(t_1t_2)\bv.
 \end{equation*}
 Moreover, if $e$ is the identity element of $T$, then 
 $$ e(\bv)_g=\psi^{-1}(e)(\bv)_g=(0,1)(\bv)_g=v_{(g-0)1}=v_g \implies e(\bv)=(0,1)(\bv)=\bv.$$
 \end{proof}
 Now, we derive some identities that we are going to use to analyze the structure of the group $D$ generated by cyclic shifts and decimations whose group operation is function composition.
 By equation~(\ref{eqn:acil}), for each $i \in \zl$, $(a_1,b_1), (a_2,b_2) \in  \zl \times \zlx$ and $\bv \in \mathbb{C}^{\zl}$ we get  
 \begin{align*}
 \begin{split}
 ((c_{a_1} \circ d_{b_1})\circ(c_{a_2} \circ d_{b_2})(\bv))_i=& v_{((i-a_1)b_1^{-1}-a_2)b_2^{-1}}=v_{(i-a_1)b_1^{-1}b_2^{-1}-a_2b_2^{-1}}\\=& \,\,v_{((i-(a_1+b_1a_2))(b_1b_2)^{-1}}
=((c_{a_1+b_1a_2} \circ d_{b_1b_2})(\bv))_i.
\end{split}
\end{align*}
Hence,  for each $(a_1,b_1), (a_2,b_2) \in  \zl \times \zlx$ and $\bv \in \mathbb{C}^{\zl}$,
\begin{equation}\label{eqn:hom}
((c_{a_1} \circ d_{b_1})\circ(c_{a_2} \circ d_{b_2}))(\bv)= (c_{a_1+b_1a_2} \circ d_{b_1b_2})(\bv).
\end{equation}
Observe that for each $\bv \in \cl$ and $i \in \zl$,
\begin{equation*}\label{eqn:addc0}
((c_{a_1}\circ c_{a_2})(\bv))_i=v_{i-a_1-a_2}=c_{a_1+a_2}(\bv)_i
\text{ for all } a_1,a_2 \in \zl
\end{equation*}
and 
\begin{equation*}\label{eqn:multd0}
((d_{k_1}\circ d_{k_2})(\bv))_i=v_{i(k_1)^{-1}(k_2)^{-1}}=d_{k_1k_2}(\bv)_i
\text{ for all } k_1,k_2 \in \zlx.
\end{equation*}
Hence,  
\begin{equation}\label{eqn:addc}
c_{a_1}\circ c_{a_2}=c_{a_1+a_2}
\text{ for all } a_1,a_2 \in \zl,
\end{equation}
\begin{equation}\label{eqn:multd}
d_{k_1}\circ d_{k_2}=d_{k_1k_2}
\text{ for all } k_1,k_2 \in \zlx,
\end{equation}
and
\begin{equation}\label{eqn:inverse}
(c_a)^{-1}=c_{-a} \text{ and } (d_k)^{-1}=d_{k^{-1}}.
\end{equation}

The following proposition shows that the group  generated by all decimations and cyclic shifts is
isomorphic to \(\sdp\).
\begin{theorem}\label{thm:isomorphic}
Let $D$ be the group generated by all decimations $d_k$ and cyclic shifts $c_j$ for $k \in \zlx$
and $j \in \zl$, where the group operation is function composition.
Then 
\begin{itemize}
\item [(i)] $D\cong \sdp$;  
\item [(ii)] $\langle (i,1) \mid (i,1)\in \sdp, i \in \zl \rangle
\cong  \langle c_i \mid i \in \zl
\rangle \cong \zl;$
\item [(iii)]
$\langle (0, k) \mid (0, k) \in \sdp, k\in \zlx \rangle\cong \langle d_k \mid k\in \zlx \rangle \cong \zlx.$
\end{itemize}
\end{theorem}
\begin{proof}
To prove (i), for $(a,b)\in \sdp$ let $\psi: \sdp \rightarrow D$ be such that
$$\psi(a,b)=c_a\circ d_b.$$
Then for $(a_1,b_1), (a_2,b_2) \in  \sdp$, by equation~(\ref{eqn:hom})
\begin{align*}
\begin{split}
\psi(a_1,b_1)\circ\psi(a_2,b_2)(\bv)=&\,\,(c_{a_1}\circ d_{b_1})\circ (c_{a_2}\circ d_{b_2})(\bv)\\=&\,\,(c_{a_1+b_1a_2} \circ d_{b_1b_2})(\bv)\\
=&\,\, \psi((a_1+b_1a_2,b_1b_2))\bv\\
=&\,\, \psi((a_1,b_1)(a_2,b_2))\bv
\end{split}
\end{align*}
for all $\bv \in \cl$.
Hence,
$$\psi(a_1,b_1)\circ \psi(a_2,b_2)= \psi((a_1,b_1)(a_2,b_2)).$$
So, $\psi$ is a homomorphism.
Let $\psi(a_1,b_1)=\psi(a_2,b_2)$.
Then $(c_{a_1}\circ d_{b_1})\bv=(c_{a_2}\circ d_{b_2})\bv$ for all $\bv \in \cl$. Hence, by equations~(\ref{eqn:addc}) and~(\ref{eqn:inverse}) we have $c_{a_1-a_2}\circ d_{b_1}(\bv)=d_{b_2}(\bv)$ for all $\bv \in \cl$.
Then, by equations~(\ref{eqn:acil}) 
$$(c_{a_1-a_2}\circ d_{b_1}(\bv))_0=v_{(-a_1+a_2)b_1^{-1}}=d_{b_2}(\bv)_0=v_0$$ for all $\bv \in \cl$.
Hence, $(-a_1+a_2)b_1^{-1}=0$ implying that $a_1=a_2$ and $d_{b_1}(\bv)=d_{b_2}(\bv)$ for all $\bv \in \cl$.
This implies
$d_{b_2^{-1}b_1}(\bv)=\bv$ for all $\bv \in \cl$ 
by equations~(\ref{eqn:multd}) and~(\ref{eqn:inverse}).
Hence $b_2^{-1}b_1=1$ and $b_1=b_2$ as by equation~(\ref{eqn:decimate}), $d_1$ is the only decimation that fixes all the vectors in $\cl$.
We have now proven that $\psi$ is injective.

By
Lemma~\ref{lem:generate} the set of all 
$(j,1)$, $(0,k)$ for  $j \in \zl$, $k \in \zlx$ generate $\sdp$.
Now, the set of all $\psi(j,1)=c_j$ and $\psi(0,k)=d_k$, and $c_j,d_k$ for $j \in \zl$ and $k \in \zlx$ generate $D$, $\psi$ is an injective homomorphism mapping a set of generators of $\sdp$ to a set of generators of $D$.
Hence, $\psi$ is onto $D$ implying that $\psi$ is an isomorphism.

To prove (ii), first the projection map $\pi_1:\langle (i,1) \mid (i,1)\in \sdp, i \in \zl
\rangle \rightarrow \zl$ such that $\pi_1 (i,1)=i$ is clearly one-to-one
and onto. Since for $i_1,i_2\in \zl$, $(i_1,1)(i_2,1)=(i_1+i_2,1)$,
we have
$$\pi_1((i_1,1)(i_2,1))=\pi_1(i_1+i_2,1)=i_1+i_2=\pi_1(i_1,1)\pi_1(i_2,1).$$
So, $\pi_1$ is an isomorphism and
$\langle (i,1) \mid (i,1)\in \sdp, i \in \zl
\rangle\cong  \zl.$ Observe that 
$\psi(\langle (i,1) \mid (i,1)\in \sdp, i \in \zl
\rangle)= \langle c_i \mid i \in \zl
\rangle$,
and we get $\langle (i,1) \mid (i,1)\in \sdp, i \in \zl \rangle
\cong  \langle c_i \mid i \in \zl
\rangle$  since $\psi$ is an isomorphism.

To prove (iii), first the projection map $\pi_2: \langle (0, k) \mid (0, k) \in \sdp, k\in \zlx \rangle \rightarrow \zlx$ such that $\pi_2 (0,k)=k$ 
is clearly one-to-one
and onto.  Since for $k_1,k_2\in \zlx$, $(0,k_1)(0,k_2)=(0,k_1k_2)$,
we have
$$\pi_2((0,k_1)(0,k_2))=\pi_2(0,k_1k_2)=k_1k_2=\pi_2(0,k_1)\pi_2(0,k_2).$$
So, $\pi_2$ is an isomorphism and
$\langle (0,k) \mid (0,k)\in \sdp, k \in \zl
\rangle\cong  \zlx.$
Observe that 
$\psi(\langle (0,k) \mid (0,k)\in \sdp, k \in \zlx
\rangle)= \langle d_k \mid k \in \zlx
\rangle$,
and we get $\langle (0,k) \mid (0,k)\in \sdp, k \in \zlx
\rangle \cong  \langle d_k \mid k \in \zlx \rangle$  since $\psi$ is an isomorphism.
\end{proof}

The next corollary shows that the group $D$ acts on $\cl$. Hence, decimation classes in  $ \mathbb{C}^{\zl}$, being the orbits of vectors in $ \mathbb{C}^{\zl}$ are indeed equivalence classes. 
\begin{corollary}
    The group $D$ generated by all cyclic shifts and all decimations with the group operation function composition acts on $\cl$.
\end{corollary}
\begin{proof}
Let  $\psi$ be as in Proposition~\ref{thm:isomorphic}. 
Then  since by proposition~\ref{thm:isomorphic} $\psi:\sdp\rightarrow D$ is an isomorphism,  
by Lemma~\ref{lem:Danc}, $D$ acts on $\cl$.  
\end{proof}
Our next goal is to show that for pairs of vectors in $\mathbb{C}^{\zl}\times\mathbb{C}^{\zl}$ being equivalent 
is an equivalence relation. To accomplish this wee need to better understand how cyclic shifts and decimations are related in $D$. In particular, 
cyclic shifts and decimations do not always commute. The following lemma provides two general relations between cyclic shifts and decimations.
\begin{lemma}\label{lem:relations}
For  \(a \in\zl\) and \(b\in\zlx\), 
$$c_a\circ d_b=d_b\circ c_{ab^{-1}} \text{ and } d_b\circ c_a=c_{ab}\circ d_b.$$
\end{lemma}
\begin{proof}
Let $\psi$ be the isomorphism in the proof of Proposition~\ref{thm:isomorphic}.
  Since 
  $\psi(a,1)=c_a$ and $\psi(0,b)=d_b$ for $a \in \zl$ and $b \in \zlx$, it suffices to prove 
   $$
   (a,1) (0,b)=(0,b) (ab^{-1},1) \text{ and } (0,b) (a,1)=(ab,1) (0,b) \,\,\,
 $$  for $(a,1), (0,b) \in \sdp$.
 Recall that for $(a_1,b_1), (a_2,b_2) \in \sdp$, $$(a_1,b_1)(a_2,b_2)=(a_1+b_1a_2,b_1b_2).$$
 Then by taking $(a_1,b_1)=(a,1)$ and $(a_2,b_2)=(0,b)$, we get
  $$
   (a,b)=(a,1)(0,b)=(0,b) (ab^{-1},1) \text{ and } (ab,b)=(0,b) (a,1)=(ab,1) (0,b).$$ 
\end{proof}
 The relation in Lemma~\ref{lem:relations} is equivalent to \(d_b\circ c_a\circ(d_b)^{-1}=d_b\circ c_a\circ d_{b^{-1}}=c_{ab}\). Hence, more generally, for \(g=c_{a'}\circ d_b\) and   \(a',a,b\in\zl\), we have 
 \begin{equation*}\label{eqn:gcaginv}
  g\circ c_a\circ g^{-1}=c_{ab}.   
 \end{equation*}
 For the rest of the paper, to simplify the notation, we do not use the ``$\circ$'' notation for compositions of cyclic shifts, decimations, switches, and any other functions. Instead, we use juxtaposition. So,  $f_1\circ f_2=f_1f_2$ for any two functions $f_1,f_2$. Since we do not multiply functions in this paper, this should not cause any confusion.

Let $(d_k,(-1)^r)\in \langle d_k \mid k\in \zlx \rangle \times \langle -1 \rangle$, $s^f \in \langle s \rangle$, $(c_i,c_j)\in \langle c_1 \rangle \times \langle c_1 \rangle$.
Then, each of the groups $\langle d_k \mid k\in \zlx \rangle \times \langle -1 \rangle$, $\langle s \rangle\cong  \langle -1 \rangle$, and 
 $ \langle c_1 \rangle \times \langle c_1 \rangle \cong \zl \times \zl$
acts on pairs of vectors $(\bu,\bv)$ such that $\bu, \bv \in \mathbb{C}^{\zl}$ 
according to 
\begin{align*}
 (d_k,(-1)^r)(\bu,\bv)&=(d_k(\bu),d_{(-1)^rk}(\bv)),\\
 (c_i,c_j)(\bu,\bv)&=(c_i(\bu),c_j(\bv)),\\
 s(\bu,\bv)&=(\bv,\bu).
\end{align*}
To see this, we first observe that
\begin{align*}
 (d_1,1)(\bu,\bv)&=(d_1(\bu),d_1(\bv))=(\bu,\bv),\\
 (c_0,c_0)(\bu,\bv)&=(c_0(\bu),c_0(\bv))=(\bu,\bv),\\
 s^0(\bu,\bv)&=(\bu,\bv).
\end{align*}
Moreover, for  $\bu, \bv \in \mathbb{C}^{\zl}$, $k_1,k_2\in \zlx$, $r_1,r_2\in \{0,1\}$, $i_1,i_2,j_1,j_2\in \zl$,
 $f_1,f_2 \in \{0,1\}$,
\begin{align*}
(d_{k_1},(-1)^{r_1})\left((d_{k_2},(-1)^{r_2})(\bu,\bv)\right)&=(d_{k_1k_2}(\bu),d_{(-1)^{r_1+r_2}k}(\bv)),\\
 (c_{i_1},c_{j_1})\left((c_{i_2},c_{j_2})(\bu,\bv)\right)&=(c_{i_1+i_2}(\bu),c_{j_1+j_2}(\bv)),\\
 s^{f_2}\left(s^{f_1}(\bu,\bv)\right)&=s^{f_1+f_2}(\bu,\bv),
\end{align*}
and
\begin{align*}
(d_{k_1},(-1)^{r_1})(d_{k_2},(-1)^{r_2})&=(d_{k_1k_2},d_{(-1)^{r_1+r_2}k}),\\
 (c_{i_1},c_{j_1})(c_{i_2},c_{j_2})&=(c_{i_1+i_2},c_{j_1+j_2}),\\
 s^{f_2}s^{f_1}&=s^{f_1+f_2}.
\end{align*}
So,
\begin{align*}
\left((d_{k_1},(-1)^{r_1})(d_{k_2},(-1)^{r_2})\right)(\bu,\bv)&=(d_{k_1k_2}(\bu),d_{(-1)^{r_1+r_2}k}(\bv)),\\
 \left((c_{i_1},c_{j_1})(c_{i_2},c_{j_2})\right)(\bu,\bv)&=(c_{i_1+i_2}(\bu),c_{j_1+j_2}(\bv)),\\
 \left(s^{f_2}s^{f_1}\right)(\bu,\bv)&=s^{f_1+f_2}(\bu,\bv).
\end{align*}
The functions $(d_k,(-1)^r)\in \langle d_k \mid k\in \zlx \rangle \times \langle -1 \rangle$, $s^f \in \langle s \rangle$, $(c_i,c_j)\in \langle c_1 \rangle \times \langle c_1 \rangle$ from $\mathbb{C}^{\zl}\times\mathbb{C}^{\zl}$ to $\mathbb{C}^{\zl}\times\mathbb{C}^{\zl}$
generate a group $GG$ under function composition. We are going to use the symbol $GG$ for this group hereafter.
In $GG$ 
we observe
\begin{equation*}\label{eqn:scc}
    s(c_i,c_j)(\bu,\bv)=(c_j,c_i)s(\bu,\bv) \text{\, and\, } s(d_k,(-1)^r)(\bu,\bv)= (d_{(-1)^rk},(-1)^r)s(\bu,\bv)
\end{equation*}
for all $(\bu,\bv) \in \mathbb{C}^{\zl}\times\mathbb{C}^{\zl}$.
Hence,
\begin{equation}\label{eqn:scc1}
    s(c_i,c_j)=(c_j,c_i)s,
\end{equation}
and
\begin{equation}\label{eqn:scc2}
 s(d_k,(-1)^r)= (d_{(-1)^rk},(-1)^r)s.
\end{equation}

Since by Lemma~\ref{lem:relations}
 \(c_a d_b=d_bc_{ab^{-1}}\),  
 and by equation~(\ref{eqn:inverse}) $d_{(-1)^rk^{-1}}=(d_{(-1)^rk})^{-1}$, 
 we get
 \begin{equation}\label{eqn:conjhalf}
\begin{aligned}
(c_i,c_j)(d_{k},(-1)^{r})(\bu,\bv)=&\,(c_i,c_j)(d_k(\bu),d_{(-1)^rk}(\bv))=(c_i(d_k(\bu)),c_j(d_{(-1)^rk}(\bv)))\\
=&\,(d_kc_{ik^{-1}}(\bu), d_{(-1)^rk}(c_{j(-1)^rk^{-1}}(\bv))).
\end{aligned}
\end{equation}
Now, observe that 
 \begin{equation}
\begin{aligned}\label{eqn:dkminvvv}
(d_{k^{-1}},(-1)^{r})(d_{k},(-1)^{r})(\bu,\bv)=&\,(d_{k^{-1}},(-1)^{r})(d_k(\bu),d_{(-1)^rk}(\bv))\\
=&\,(d_{k^{-1}k}(\bu),d_{(-1)^{2r}k^{-1}k}(\bv))=(\bu,\bv)
\end{aligned}
\end{equation}
for all $(\bu,\bv) \in \mathbb{C}^{\zl}\times  \mathbb{C}^{\zl}$.
Hence, $(d_{k},(-1)^{r})^{-1}=(d_{k^{-1}},(-1)^{r})$.
Moreover, by equations~(\ref{eqn:conjhalf}) and~(\ref{eqn:dkminvvv}) 
 \begin{equation*}\label{eqn:conjugatedk}
\begin{aligned}
& (d_{k},(-1)^{r})^{-1}(c_i,c_j)(d_{k},(-1)^{r})(\bu,\bv)\\=&\,(d_{k^{-1}},(-1)^{r})(d_kc_{ik^{-1}}(\bu), d_{(-1)^rk}(c_{j(-1)^rk^{-1}}(\bv)))\\
=&\,(d_{{k^{-1}}k}c_{ik^{-1}}(\bu), (d_{k^{-1}(-1)^{2r}k}(c_{j(-1)^rk^{-1}}(\bv)))\\
=&\,(c_{ik^{-1}}(\bu), c_{j(-1)^rk^{-1}}(\bv)).
\end{aligned}
\end{equation*}
Hence,
 \begin{equation}\label{eqn:conjfinal}
 (d_{k},(-1)^{r})^{-1}(c_i,c_j)(d_{k},(-1)^{r})=(c_{ik^{-1}}, c_{j(-1)^rk^{-1}}).
\end{equation}
Now, $GG$ acts on  paired vectors $(\bu,\bv)$
 with $\bu,\bv\in \mathbb{C}^{\zl}$ because the relations among the generators of $GG$
 are completely determined by how the elements of the generating subgroups  $\langle d_k \mid k\in \zlx \rangle \times \langle -1 \rangle$, $\langle s \rangle$, and 
 $ \langle c_1 \rangle \times \langle c_1 \rangle$ act on paired vectors.
 This is true because the group operation in $GG$ is function composition.
 
The next lemma defines the conjugation automorphism $\sigma_k$. This is different  from how $\sigma_k$ was defined earlier. It will be needed in proving  a subsequent proposition that generalizes Proposition~\ref{thm:isomorphic}. The generalization of Proposition~\ref{thm:isomorphic} will be needed to determine the structure of $GG$.
\begin{lemma}\label{lem:nataut}
For a group $G$ and a normal subgroup $K$ ($K\mathrel{\unlhd}G$) and $Q\leq G$, let
$\sigma_q:G\rightarrow G$ be
 such that $\sigma_q(g)=qgq^{-1}$ for each $g \in G$ and a fixed $q \in Q$. Then
 \begin{itemize}
 \item[(i)]
 $\sigma_q\in \aut(G)$ for each fixed $q \in Q$;
 \item[(ii)]$\theta:Q\rightarrow \aut(K)$ such that
 $\theta(q)=\sigma_q$ for $ q \in Q$ is a homomorphism.
 \end{itemize}
 \end{lemma}
 \begin{proof}
    To see (i), observe that for each fixed $q \in Q$ and each $g \in G$,  
    $$\sigma_{q^{-1}}(\sigma_q)(g)=q^{-1}qgq^{-1}(q^{-1})^{-1}=g$$
    and 
     $$\sigma_q(\sigma_{q^{-1}})(g)=qq^{-1}g(q^{-1})^{-1}(q)^{-1}=g.$$
     Hence, $\sigma_q$ is invertible. Moreover, for $g_1,g_2\in G$,
     $$\sigma_q(g_1g_2)=q(g_1g_2)q^{-1}=q(g_1)q^{-1}q(g_2)q^{-1}= \sigma_q(g_1)\sigma_q(g_2)$$
     implying that $\sigma_q$ is an automorphism of $G$.
     To prove (ii), let $\theta(q) = \sigma_q \arrowvert_K$ ($\sigma_q$ restricted to $K$). By the normality of $K$, note that
$\sigma_q(K) = K$ so that $\theta(q)$ is well defined and an automorphism on $K$. We
next show $\theta$ is a homomorphism.
     For $q_1,q_2\in Q$, observe that
     $$(\theta(q_1)\circ\theta(q_2))(k)=(\sigma_{q_1}\circ\sigma_{q_2})(k)=
     \sigma_{q_1}(q_2kq_2^{-1})=q_1q_2kq_2^{-1}q_1^{-1}=\sigma_{q_1q_2}(k)=\theta(q_1q_2)(k).$$   
 \end{proof}
The next proposition  generalizes Proposition~\ref{thm:isomorphic}.
\begin{theorem}\label{thm:semidirect}
For a  group $G$ with identity element $id$, let $K\mathrel{\unlhd}G$, $Q\leq G$, and $\langle K,Q\rangle=G$. 
 Let $\theta:Q\rightarrow \aut(K)$ be defined by $\theta(q) =\sigma_q \arrowvert_K$
and $\sigma_q(k) = q k q^{-1}$. Then 
 \begin{itemize}
     \item[(i)]  $KQ=\langle K,Q\rangle=G$;
     \item[(ii)] $KQ=QK$;
     \item[(iii)] if $G$ is finitely generated and  $K\cap Q =\{id\}$, then $G \cong K\rtimes_{\theta} Q$.
 \end{itemize}
\end{theorem}
\begin{proof}
To prove part (i), let  $(a_1,b_1), (a_2,b_2) \in  K \times Q$.
Since  $K\mathrel{\unlhd}G$, $$(a_1b_1)(a_2b_2)=a_1b_1a_2b_1^{-1}b_1b_2=a_1a_2'b_1b_2,$$
 where $b_1a_2b_1^{-1}=a_2'\in K$. Hence, $KQ$ is closed under the group operation of $G$.
Moreover, for  $(a,b) \in  K \times Q$ we have 
$$(ab)^{-1}=b^{-1}a^{-1}=b^{-1}a^{-1}bb^{-1}=a''b^{-1},$$ where 
$b^{-1}a^{-1}b=b^{-1}a^{-1}(b^{-1})^{-1}=a''\in K$ as $K\mathrel{\unlhd}G$. Hence, 
$(ab)^{-1} \in KQ$. This shows that the set $KQ$ is a group under the group operation of $G$.
It is also clear that $K\leq KQ$ and  $Q\leq KQ$. 
Hence, $ \langle K,Q \rangle \leq KQ $. On the other hand  $\langle K,Q \rangle$ contains every element of the form $ab$ for $(a,b)\in K \times Q$. Hence,  $G =\langle K,Q \rangle = KQ. $

The proof of part (ii) mimics the  proof of part (i).
Let  $(b_1,a_1), (b_2,a_2) \in  Q \times K$.
Then,   since  $K\mathrel{\unlhd}G$, $$(b_1a_1)(b_2a_2)=b_1b_2b_2^{-1}a_1b_2a_2=b_1b_2a_1'a_2,$$
 where $b_2^{-1}a_1(b_2^{-1})^{-1}=b_2^{-1}a_1b_2=a_1'\in K$. Hence, $QK$ is closed under the group operation of $G$.
Moreover, for  $(b,a) \in  Q \times K$, we have 
$$(ba)^{-1}=a^{-1}b^{-1}=a^{-1}b^{-1}aa^{-1}=a''b^{-1},$$ where
$a^{-1}b^{-1}a=a^{-1}b^{-1}(a^{-1})^{-1}=a''\in K$ as $K\mathrel{\unlhd}G$. Hence, 
$(ba)^{-1} \in QK$. This shows that the set $QK$ is a group under the group operation of $G$.
It is also clear that $K\leq QK$ and  $Q\leq QK$. 
Hence, $ \langle K,Q \rangle \leq QK $. On the other hand  $\langle K,Q \rangle$ contains every element of the form $ba$ for $(b,a)\in Q \times K$. Hence,  $G =\langle K,Q \rangle = QK. $

The proof of part (iii) mimics the  part (i) proof of Proposition~\ref{thm:isomorphic}.
For $(a,b)\in \sdpq$ let $\psi: \sdpq \rightarrow KQ=G$ be such that
$$\psi(a,b)=ab.$$
Then for $(a_1,b_1), (a_2,b_2) \in  \sdpq$, 
\begin{align*}
\begin{split}
\psi(a_1,b_1)\psi(a_2,b_2)=&\,\,(a_1b_1)(a_2b_2)=a_1b_1a_2b_1^{-1}b_1b_2 \\ =&\,\, a_1\sigma_{b_1}(a_2)b_1b_2 = a_1\theta(b_1)(a_2)b_1b_2\\
=&\,\, \psi((a_1,b_1)(a_2,b_2)).
\end{split}
\end{align*}
Hence,
$$\psi(a_1,b_1)\psi(a_2,b_2)= \psi((a_1,b_1)(a_2,b_2)).$$
So, $\psi$ is a homomorphism.

Let $\psi(a_1,b_1)=\psi(a_2,b_2)$.
Then $$a_1 b_1=a_2 b_2
\implies a_2^{-1}a_1=b_2b_1^{-1}\in  K\cap Q=\{id\}.$$
Hence, $a_1=a_2$ and $b_1=b_2$ implying that $\psi$ is injective. 

By
Lemma~\ref{lem:generate} the set of all 
$(j,id)$, $(id,k)$ for  $j \in K$, $k \in Q$ generates $\sdpq$.
Now, $\psi(j,id)=j$ and $\psi(id,k)=k$ and $j,k$ generate $G=KQ$ for $j \in K$ and $k \in Q$. Then $\psi$ is an injective homomorphism mapping a set of generators of $\sdpq$ to a set of generators of $G=KQ$.
Hence, $\psi$ is onto $G=KQ$ implying that $\psi$ is an isomorphism.
\end{proof}
Hereafter, for a semidirect product where $\theta$ is defined by $\theta(q) = \sigma_q \arrowvert_K$ and $\sigma_q(k) = q k q^{-1}$, we are going to use ``$\rtimes$" instead of ``$\rtimes_{\theta}$".

We recall the definition of $GG$ and determine the structure of $GG$ in the next proposition. 
\begin{theorem}\label{thm:new}
Let $id$ be the identity element in $GG$. Let
    $G_1=\langle c_1\rangle \times \langle c_1\rangle,$
$G_2= \left(\langle d_k \mid k\in \zlx \rangle \times \langle -1 \rangle\right),$
and 
$G_3=\langle s \rangle.$ Let $GG = \langle G_1, G_2, G_3 \rangle$. 
Then
 \begin{align*}
GG\cong &\,\, G_1\rtimes\left(G_2\rtimes G_3 \right).
\end{align*} 
 \end{theorem}
\begin{proof}    
We first show that   $G_2\cap G_3 =\{id\}$ in $GG$.
Let $g\in G_2\cap G_3$. Then $$g=s^f=(d_k,(-1)^{r})$$ for some $f,r \in \{0,1\}$ and $k \in \zlx$.
Then $$g(\bu,\bv)=s^f(\bu,\bv)=(d_k,(-1)^{r})(\bu,\bv)$$ for all $(\bu,\bv)\in \mathbb{C}^{\zl}\times  \mathbb{C}^{\zl}$. However,  for $(\bx,\by)=(d_k,(-1)^{r})(\bu,\bv)$, $x_0=u_0$ and $y_0=v_0$ for all $k \in \zlx$ and 
$r \in \{0,1\}$. 
Whereas for $(\bx,\by)=s^f(\bu,\bv)$, $x_0=u_0$, and $y_0=v_0$ imply $f=0$ by picking a $(\bu,\bv)$ such that $u_0 \neq v_0$. Hence $g = s^0 = id$ in $GG$ and
$G_2\cap G_3 =\{id\}$ in $GG$.
Next, we observe  that by equation~(\ref{eqn:scc2}), $G_2 \mathrel{\unlhd} \langle G_2,G_3\rangle $.
Hence, since  $\langle G_2,G_3\rangle$ is finitely generated, by Proposition~\ref{thm:semidirect} parts (i) and (iii),  $G_2G_3= \langle G_2,G_3\rangle \cong G_2\rtimes G_3$.
Now, we show that   $$G_1\cap \langle G_2,G_3\rangle  =\{id\}.$$ Let $g \in G_1\cap \langle G_2,G_3\rangle$.
Then $$g=(c_i,c_j)=(d_k,(-1)^{r})s^f$$ for some $i,j \in \zl$, $k \in \zlx$, and $r,f \in \{0,1\}.$ 
Then $$g(\bu,\bv)=(c_i,c_j)(\bu,\bv)=(d_k,(-1)^{r})(s^f(\bu,\bv))$$ for all $(\bu,\bv)\in \mathbb{C}^{\zl}\times  \mathbb{C}^{\zl}$. However,  if $(\bx,\by)=(d_k,(-1)^{r})(s^f(\bu,\bv))$, then $x_0\in \{u_0,v_0\}$ and $y_0\in \{u_0,v_0\}$. 
Whereas, for $(\bx,\by)=(c_i,c_j)(\bu,\bv)$, $x_0\in \{u_0, v_0\}$ and $y_0\in \{u_0, v_0\}$ imply $i=0$ and $j=0$
after picking $\bu=(-1,1,\ldots,1)$ and $\bv=(-1,1,\ldots,1)$. Hence $g=id$ and
$G_1\cap \langle G_2,G_3\rangle =\{id\}$ in $GG$.
Next, we observe that by equations~(\ref{eqn:scc1}) and~(\ref{eqn:conjfinal})
 $G_1 \mathrel{\unlhd} \langle G_1, \langle G_2,G_3\rangle \rangle= \langle G_1, G_2G_3\rangle $.
 Moreover, $$GG= \langle G_1,  G_2,G_3 \rangle =\langle G_1, \langle G_2,G_3\rangle \rangle= \langle G_1, G_2G_3\rangle.$$
 Hence, since  $ \langle G_1, G_2,G_3 \rangle $ is finitely generated, by Proposition~\ref{thm:semidirect} parts (i) and (iii),  $$GG= \langle G_1, G_2G_3\rangle =G_1G_2G_3\cong G_1\rtimes G_2G_3\cong G_1\rtimes (G_2 \rtimes G_3).$$
Hence, we get $$GG=G_1G_2G_3\cong  G_1\rtimes (G_2 \rtimes G_3).$$ 
 


\end{proof}
The next proposition shows that  equivalence, as defined  by
equations~(\ref{eqn:equiv1}) and~(\ref{eqn:equiv2}),  is indeed an equivalence relation.
\begin{theorem}
Two pairs of vectors  \((\bu,\bv), (\bu',\bv') \in \cl \times \cl\) are {\em equivalent}  as defined  by
equations~(\ref{eqn:equiv1}) and~(\ref{eqn:equiv2})  if 
\((\bu,\bv)\) and \((\bu',\bv')\) are in the same orbit under the action of $GG$.
Hence, equivalence as defined  by equations~(\ref{eqn:equiv1}) and~(\ref{eqn:equiv2}) is indeed an equivalence relation.
\end{theorem}
\begin{proof}
We first use  Proposition~\ref{thm:semidirect} parts (i) and (ii) with $K = G_1$ and $Q=\langle G_2,G_3\rangle = G_2 G_3 = G_3 G_2$. Then we use the  associativity of function composition. Observe 
\begin{equation}\label{eqn:G1G2G3}
\begin{aligned}
GG=\langle G_1, G_2,G_3\rangle
=&\,G_1\langle G_2,G_3\rangle=G_1(G_2G_3)=G_1(G_3G_2)\\ =&\,(G_3G_2)G_1=G_3G_2G_1.
\end{aligned}
\end{equation}

By equations~(\ref{eqn:G1G2G3}), for any $g \in GG$
and a pair of vectors \((\bu,\bv)\) 
\begin{equation}\label{eqn:G1}
g(\bu,\bv)=s^f\Big(d_k\big(c_i(\bu)\big),d_{k}\big(c_j(\bv)\big)\Big)
\end{equation}
or 
\begin{equation}\label{eqn:G2}
g(\bu,\bv)=s^f\Big(d_k\big(c_i(\bu)\big),d_{-k}\big(c_j(\bv)\big)\Big)
\end{equation}
for some \(k\in\zlx\), \(i,j\in\zl\), and \(f\in\{0,1\}\).
Then  by equations~(\ref{eqn:G1}) and~(\ref{eqn:G2}), two pairs of vectors  \((\bu,\bv), (\bu',\bv') \in \cl \times \cl\) are {\em equivalent}  as defined in 
equations~(\ref{eqn:equiv1}) and~(\ref{eqn:equiv2}) if 
\((\bu,\bv)\) and \((\bu',\bv')\) are in the same orbit under the action of $GG$.
\end{proof}
\bmhead{Acknowledgments}

  The views expressed in this article are those of the authors, and do not reflect the official policy or position of the United States Air Force, Department of Defense, or the U.S.~Government. 
  
Data Deposition Information: No data sets have been used.





\bibliography{library}

\end{document}